\documentclass[12pt]{amsart}  

\usepackage{custom_preamble} 

\begin{document}

\title{Popular difference sets}

\author{Tom Sanders}
\address{Department of Pure Mathematics and Mathematical Statistics\\
University of Cambridge\\
Wilberforce Road\\
Cambridge CB3 0WB\\
England } \email{t.sanders@dpmms.cam.ac.uk}

\begin{abstract}
We provide further explanation of the significance of an example in a recent paper of Wolf in the context of the problem of finding large subspaces in sumsets. 
\end{abstract}

\maketitle

\setcounter{section}{1}

\section*{}

Throughout these notes $G$ will denote the group $\F_2^n$, that is the $n$-dimensional vector space over $\F_2$, and $\P_G$ will denote the normalized counting measure on $G$. The convolution of two functions $f,g \in L^1(\P_G)$ is defined to be
\begin{equation*}
f \ast g (x):=\int{f(y)g(x-y)d\P_G(y)},
\end{equation*}
and is of particular importance in additive combinatorics. To see why recall that if $A,B \subset G$ then we write $A+B$ for the sumset $\{a+b:a \in A, b \in B\}$. Convolution then yields the following important identity
\begin{equation*}
A+B = \supp 1_A \ast 1_B.
\end{equation*}
The sumset $A+B$ can thus be studied through the convolution $1_A \ast 1_B$; a typical example of this is in the following problem.
\begin{problem}[Bourgain-Green, {\cite{bou::4,gre::0,gre::9}}]\label{pblm}
Suppose that $A \subset G:=\F_2^n$ has density $\Omega(1)$. How large a subspace $V$ with $V \subset A+A$ can we guarantee may be found?
\end{problem}
All known approaches to this problem proceed by harmonic analysis of $1_A \ast 1_A$. Since the arguments are analytic the methods are not good at distinguishing between when $1_A \ast 1_A$ is small and when it is very small. To be more precise we introduce a definition.

Given a set $A \subset G$ of density $\alpha>0$ and a parameter $c \in [0,1]$ we define the \emph{popular difference set} with parameter $c$ to be
\begin{equation*}
D_c(A):=\{x \in G: 1_A \ast 1_A(x)>c\alpha^2\}.
\end{equation*}
Our definition employs a slightly different normalization to those presented elsewhere so that $c$ naturally lies in the range $[0,1]$. Indeed, it is easy to see that if $c$ is greater than $1$ then even for large sets $A$ we may have $D_c(A)=\{0_G\}$.\footnote{ In particular this occurs generically: if $c>1$, $|G|$ is large enough in terms of $c$, and $A$ is chosen uniformly at random from sets of size $|G|/2$ then (with high probability) $D_c(A)=\{0_G\}$.}

At the other end of the spectrum we have $D_0(A)=A+A$. Now, the arguments for tackling Problem \ref{pblm} all fail to meaningfully distinguishing between the set $A+A=D_0(A)$ and $D_c(A)$ when $c$ is small. In particular, for example, Green effectively proves the following result in \cite{gre::9}.
\begin{theorem}\label{thm.subthm}
Suppose that $G:=\F_2^n$, $A \subset G$ has density $\alpha= \Omega(1)$ and $c \in [0,1]$ is a parameter. Then there is a subspace $V \subset D_c(A)$ with
\begin{equation*}
|V| \geq  \exp(\Omega((1-c) n)).
\end{equation*}
\end{theorem}
Notice that the bound necessarily tends to $1$ as $c$ tends to $1$, but when $c$ is small there is no significant variation. Setting $c=0$ one has the following corollary and nothing stronger is known.
\begin{corollary}[{\cite[Theorem 9.3]{gre::9}}]\label{cor.triv}
Suppose that $G:=\F_2^n$ and $A \subset G$ has density $\alpha= \Omega(1)$. Then there is a subspace $V \subset A+A$ with
\begin{equation*}
|V| \geq  \exp(\Omega(n)).
\end{equation*}
\end{corollary}

In the other direction, adapting a construction from Ruzsa \cite{ruz::6} (see also \cite{ruz::4}), Green showed the following result.
\begin{theorem}[{\cite[Theorem 9.4]{gre::9}}]\label{thm.nivtriv}
Suppose that $G:=\F_2^n$. Then there is a set $A \subset G$ with density $\alpha=\Omega(1)$ such that if $V \subset A+A$ is a subspace then
\begin{equation*}
\P_G(V) \leq \exp(-\Omega(\sqrt{n})).
\end{equation*}
\end{theorem}
Of course the gap between Corollary \ref{cor.triv} and Theorem \ref{thm.nivtriv} is very large so in a sense Problem \ref{pblm} remains wide open.

Since our only ways of proving that $A+A$ contains a subspace also show that $D_c(A)$ contains a subspace it is natural to ask about the limitations of such methods and, in particular, ask Problem \ref{pblm} with $D_c(A)$ in place of $A+A$.

It turns out that $D_c(A)$ need not contain a large subspace. Indeed, more than this, in the paper \cite{wol::0} Wolf was able to show that there are sets where not only does $D_c(A)$ not contain a large subspace, it doesn't even contain the sumset of a large set.\footnote{In actual fact the bulk of \cite{wol::0} concerns the arithmetic case $\Z/N\Z$; our purpose is expository and so our interest remains with $\F_2^n$.}
\begin{theorem}[{\cite[Theorem 2.6]{wol::0}}]\label{thm.wolf}
Suppose that $G:=\F_2^n$ and $c \in (0,1/2]$ is a parameter. Then there is a set $A=A(c) \subset G$ with density $\alpha=\Omega(1)$ such that if $A'+A' \subset D_c(A)$ then
\begin{equation*}
\P_G(A') \leq \exp(-\Omega(n/\log^2 c^{-1})).
\end{equation*}
\end{theorem}
We remind the reader that we are interested in the case when $c$ is small but fixed; as $c$ tends to $0$ the upper bound must tend to $1$ since $A+A = D_0(A)$.

The proof makes appealing use of measure concentration in $G$. Indeed, it is a key insight here to consider the more general question of containing the sumset of a large set, not just a subspace, as this is much more suggestive of such tools.
\begin{problem}[Wolf, \cite{wol::0}]
Suppose that $A \subset G:=\F_2^n$ has density $\Omega(1)$ and $c \in [0,1]$ is a parameter. How large a set $A'$ with $A'+A' \subset D_c(A)$ can we guarantee may be found?
\end{problem}
Remarkably it turns out that it is quite easy to see that Wolf's theorem is close to best possible; we have the following complementary result.
\begin{theorem}\label{thm.popdiffthm}
Suppose that $G:=\F_2^n$, $A \subset G$ has density $\alpha= \Omega(1)$ and $c \in (0,1]$ is a parameter. Then there is a set $A'$ such that $A'+A' \subset D_c(A)$ and
\begin{equation*}
\P_G(A') \geq \exp(-O(n/\log c^{-1})).
\end{equation*}
\end{theorem}
This result does not quite achieve the bounds of Theorem \ref{thm.wolf} and it seems of interest to try to close this gap.

Since we are looking for the considerably weaker structure of a sumset rather than a subspace we have some rather stronger tools available to us in the form of Gowers' \cite{gow::4} proof of the Balog-Szemer{\'e}di theorem \cite{balsze::}. We shall prove the following explicit version of Theorem \ref{thm.popdiffthm}
\begin{theorem}\label{thm.popdiffthmact}
Suppose that $G:=\F_2^n$, $A \subset G$ has density $\alpha>0$ and $c \in (0,1/2]$ is a parameter. Then there is a set $A'$ such that $A'+A' \subset D_c(A)$ and
\begin{equation*}
|A'| \geq \lfloor \alpha^32^{n(1-\log \alpha^{-1}/\log c^{-1})}/12\rfloor.
\end{equation*}
\end{theorem}
Note that, as expected, Gowers' method provides rather good density dependence; using Fourier methods one might expect this to be exponential in $\alpha^{-O(1)}$ rather than in $\log \alpha^{-1}$.

The following result is an obvious generalization of the `combinatorial lemma' of Gowers \cite[Lemma 11]{gow::4}.
\begin{lemma}\label{lem.combinlem}
Suppose that $G:=\F_2^n$, $A \subset G$ has density $\alpha>0$ and $c,\sigma \in (0,1]$ are parameters. Then there is a set $A' \subset G$ with 
\begin{equation*}
\P_G(A') \geq \alpha^{\lceil \log 2\sigma^{-1}/\log c^{-1}\rceil}/\sqrt{2},
\end{equation*}
and
\begin{equation*}
\P_G^2(\{(x,y) \in A'^2: x+y \in D_c(A)\}) \geq (1-\sigma)\P_G^2(A'^2).
\end{equation*}
\end{lemma}
\begin{proof}
Let $r$ be a natural parameter to be specified later and $X_1,\dots,X_r$ be independent, uniform, $G$-valued random variables. Put $A_i=X_i+A$ and $A'=\bigcap_{i=1}^r{A_i}$.

The probability that $(x,y) \in A_i^2$ is $1_A \ast 1_A(x+y)$, so the probability that $(x,y) \in A'^2$ is $1_A \ast 1_A(x+y)^r$. However, $\E_{x,y \in G}{1_A \ast 1_A(x+y)} = \alpha^2$ whence, by H{\"o}lder's inequality, we have that $\E\P_G(A')^2 =\E_{x,y \in G}{1_A \ast 1_A(x+y)^r} \geq \alpha^{2r}$.

Let $S$ be the set of pairs $(x,y) \in A'^2$ such that $x+y \not \in D_c(A)$. Then, as before, the probability that $(x,y) \in S$ is $1_A \ast 1_A(x+y)^r \leq c^r\alpha^{2r}$ and it follows that the expectation of $\P_G^2(A'^2) - \sigma^{-1}\P_G^2(S)$ is at least $\alpha^{2r}(1-\sigma^{-1}c^r)$. Letting $r=\lceil \log 2\sigma^{-1}/\log c^{-1}\rceil$ we conclude that there are some values of $X_1,\dots,X_r$ such that $\P_G^2(A'^2) - \sigma^{-1}\P_G^2(S) \geq \alpha^{2r}/2$. In particular
\begin{equation*}
\P_G^2(A'^2) - \sigma^{-1}\P_G^2(S)\geq 0 \textrm{ and }\P_G^2(A'^2)  \geq \alpha^{2r}/2,
\end{equation*}
and the result follows.
\end{proof}
We may now prove the theorem using the pigeon-hole principle.
\begin{proof}[Proof of Theorem \ref{thm.popdiffthmact}]
Let $\sigma$ be a parameter such that
\begin{equation}\label{eqn.restrict}
\left\lceil |G| \alpha^{\lceil \log 2\sigma^{-1}/\log c^{-1}\rceil}/\sqrt{2} \right\rceil \geq  \sigma^{-1}.
\end{equation}
Apply Lemma \ref{lem.combinlem} with $\sigma$ to get a set $A_0 \subset G$ with $|A_0| \geq  \sigma^{-1}$ and
\begin{equation}\label{eqn.key}
\P_G^2(\{(x,y) \in A_0^2: x+y \in D_c(A)\}) \geq (1-\sigma)\P_G^2(A_0^2).
\end{equation}
Let $A_1 \subset A_0$ be chosen uniformly at random from sets of size $\lfloor \sigma^{-1}/4 \rfloor$ (possible since $|A_0|$ is large enough). It follows that $\E|\{(x,y) \in A_1\times A_1: x+y \in D_c(A)\}|$ is equal to
\begin{equation*}
\E|\{(x,y) \in A_1\times A_1: x \neq y, x+y \in D_c(A)\}|+|A_1|.
\end{equation*}
The first term here is at least
\begin{equation*}
\frac{|A_1|(|A_1|-1)}{|A_0|(|A_0|-1)}((1-\sigma)|A_0|^2-|A_0|),
\end{equation*}
by (\ref{eqn.key}). Whence
\begin{equation*}
\E|\{(x,y) \in A_1\times A_1: x+y \in D_c(A)\}| \geq (1-2\sigma)|A_1|^2,
\end{equation*}
since $|A_0|^{-1} \leq \sigma$. We may thus pick a set $A_1$ such that this inequality is satisfied.

Now, let
\begin{equation*}
A_2:=\{x \in A_1: |\{ y \in A_1: x+y \in D_c(A)\}| \geq (1-3\sigma)|A_1|\}.
\end{equation*}
Note that since $3\sigma |A_1| < 1$, $x \in A_2$ implies that $x+A_1 \subset D_c(A)$. However, $A_2 \subset A_1$ whence $A_2+A_2 \subset D_c(A)$. It remains to note that
\begin{equation*}
|A_2||A_1| + (|A_1|-|A_2|)(1-3\sigma)|A_1| \geq (1-\sigma)|A_1|^2,
\end{equation*}
whence $|A_2| \geq |A_1|/2 \geq \lfloor \sigma^{-1}/4\rfloor /2$ and we have the conclusion on maximizing this with $\sigma$ subject to (\ref{eqn.restrict}).
\end{proof}

\bibliographystyle{alpha}

\bibliography{references}

\begin{thebibliography}{Gow98}

\bibitem[Bou90]{bou::4}
J.~Bourgain.
\newblock On arithmetic progressions in sums of sets of integers.
\newblock In {\em A tribute to Paul Erd\H os}, pages 105--109. Cambridge Univ.
  Press, Cambridge, 1990.

\bibitem[BS94]{balsze::}
A.~Balog and E.~Szemer{\'e}di.
\newblock A statistical theorem of set addition.
\newblock {\em Combinatorica}, 14(3):263--268, 1994.

\bibitem[Gow98]{gow::4}
W.~T. Gowers.
\newblock A new proof of {S}zemer\'edi's theorem for arithmetic progressions of
  length four.
\newblock {\em Geom. Funct. Anal.}, 8(3):529--551, 1998.

\bibitem[Gre02]{gre::0}
B.~J. Green.
\newblock Arithmetic progressions in sumsets.
\newblock {\em Geom. Funct. Anal.}, 12(3):584--597, 2002.

\bibitem[Gre05]{gre::9}
B.~J. Green.
\newblock Finite field models in additive combinatorics.
\newblock In {\em Surveys in combinatorics 2005}, volume 327 of {\em London
  Math. Soc. Lecture Note Ser.}, pages 1--27. Cambridge Univ. Press, Cambridge,
  2005.

\bibitem[Ruz87]{ruz::4}
I.~Z. Ruzsa.
\newblock Essential components.
\newblock {\em Proc. London Math. Soc. (3)}, 54(1):38--56, 1987.

\bibitem[Ruz91]{ruz::6}
I.~Z. Ruzsa.
\newblock Arithmetic progressions in sumsets.
\newblock {\em Acta Arith.}, 60(2):191--202, 1991.

\bibitem[Wol10]{wol::0}
J.~Wolf.
\newblock The structure of popular difference sets.
\newblock {\em Israel J. Math.}, 179:253--278, 2010.

\end{thebibliography}

\end{document}